\documentclass[12pt]{amsart}
\usepackage{amssymb}
\usepackage{color}
\usepackage{amsthm}
\usepackage[pdftex]{graphicx}
\usepackage[margin=1in]{geometry}  

\newcommand{\grad}{\nabla}
\newcommand{\Laplacian}{\Delta}
\DeclareMathOperator{\diver}{div}
\DeclareMathOperator{\curl}{curl}
\DeclareMathOperator{\Tr}{Tr}
\DeclareMathOperator{\Jac}{Jac}

\newcommand{\Diff}{\mathrm{Diff}}
\newcommand{\Diffmu}{\Diff_{\mu}}
\newcommand{\id}{\mathrm{id}}

\newcommand{\llangle}{\langle\!\langle}
\newcommand{\rrangle}{\rangle\!\rangle}
\newcommand{\transpose}{\dagger}
\newtheorem{theorem}{Theorem}[section]
\newtheorem{proposition}[theorem]{Proposition}
\newtheorem{lemma}[theorem]{Lemma}

\begin{document}

\title{Lagrangian aspects of the axisymmetric Euler equation}

\author{Stephen C. Preston}
\address{Department of Mathematics, University of Colorado, Boulder, CO 80309-0395}
\email{Stephen.Preston@colorado.edu}
\author{Alejandro Sarria}
\address{Department of Mathematics, Williams College, Williamstown, MA 01267}
\email{Alejandro.Sarria@williams.edu}

\thanks{This research was supported in part by NSF grant \#1105660 and Simons Foundation grant \#1551547.}

\begin{abstract}
In this paper we are interested in geometric aspects of blowup in the axisymmetric 3D Euler equations with swirl on a cylinder.
Writing the equations in Lagrangian form for the flow derivative along either the axis or the boundary and imposing
oddness on the vertical component of the flow, we extend some blowup criteria due to Chae, Constantin, and Wu related
to assumptions on the sign of the pressure Hessian. In addition we give a geometric interpretation of the results,
both in terms of the local geometry along trajectories
and in terms of the Riemannian geometry of the volume-preserving diffeomorphism group. 
\end{abstract}

\maketitle

\section{Introduction}

The question of whether smooth solutions of the three-dimensional Euler equations can break down in finite time is a long-standing
open problem; see Constantin~\cite{constantinreview} for a history and survey of results. Even in the axisymmetric case,
where the velocity components do not depend on the angular coordinate $\theta$, the question is still open, although if
in addition the angular velocity is assumed to be zero (that is, axisymmetric flow without swirl), global existence is
well-known~\cite{majdabertozzi}. Numerical simulations of Luo-Hou~\cite{houluo} suggest very
strongly that axisymmetric solutions can blow up: their model features initial data where both the vertical and
angular velocities are odd in the vertical coordinate $z$, and they observe numerically a blowup at a fixed point
on the boundary. The increasingly common view among experts~\cite{taoblog, houconvection} is that functional analysis estimates
are not sufficient to establish blowup, and instead one must analyze the geometry of trajectories in a careful way and
use the special features of the Euler equation. We believe in addition that the Riemannian structure of the equation,
as a geodesic equation on the group of volumorphisms as found by Arnold~\cite{arnold}, is
also quite useful in both analyzing the local geometry and in finding simpler models with the same behavior. 

Our configuration manifold is the solid torus $M = D\times S^1 = \{(x,y,z) \, \vert \, x^2+y^2\le 1, -\pi \le z\le \pi\}/\{(x,y,-\pi)\equiv (x,y,\pi)\}$. The Euler equation for an axisymmetric velocity field $U$ which is tangent to $\partial M$ is given by
\begin{equation}\label{eulergeneral}
U_t + U\cdot \nabla U = -\grad P, \qquad \diver{U} = 0, \qquad \partial_{\theta}U = 0, \qquad \langle U,e_r\rangle_{r=1} = 0, \qquad U_{t=0} = U_0,
\end{equation}
where the pressure $P$ is determined nonlocally by the equation
\begin{equation}\label{pressuregeneral}
\Laplacian P = -\diver{(U\cdot \nabla U)} = -\Tr{(\nabla U)^2}, \qquad \langle \grad P, e_r\rangle_{r=1} = -\langle U\cdot \nabla U, e_r\rangle_{r=1}.
\end{equation}
The vorticity $\omega = \curl{U}$ satisfies the conservation law
\begin{equation}
\label{vorticity}
\omega_t + [U,\omega] = 0.
\end{equation}
It has been known since the work of Ebin and Marsden \cite{ebinmarsden} that the Euler equations have $H^s$ solutions for $s>5/2$ on some time interval $[0,T)$ for 
initial data $U_0$ in $H^s$.
The most famous criterion for extensibility of solutions past time $T$ is the Beale-Kato-Majda result~\cite{BKM} that
\begin{equation}\label{BKM}
\int_0^T \lVert \omega(t)\rVert_{L^{\infty}} \, dt <\infty.
\end{equation}

Expressing $U$ in components as $U = ae_r + be_{\theta} + ce_z$, we observe in Proposition
\ref{oddeven} that if initially $c$ is odd in $z$, $a$ is even in $z$, and $b$ is either even or odd in $z$,
then the same remains true as long as the solution persists. In either case the points where $z=0$ and either
$r=0$ or $r=1$ become fixed points. Criteria for either blowup or global existence in a neighborhood of these
fixed points have been studied by Chae~\cite{chae1}\cite{chae2} and by \cite{chaeconstantinwu} in the case where
$b$ is odd, but to date we do not know of any other analysis of the case where $b$ is even; as the latter case allows the
vorticity to be nonzero at the fixed point, we believe it is quite interesting. For example, we prove the following theorem
that severely restricts the possibilities for blowup at the origin when the angular velocity is even and nonzero.

\begin{theorem}\label{alpharswirlaxistheorem}
Suppose $c_0$ is odd in $z$, and that $a_0$ and $b_0$ are both even in $z$. Then the origin is a fixed point of the flow \eqref{flowequation}.
Assume that the solution of \eqref{eulergeneral} ceases to exist at time $t=T$ due to pointwise blowup of vorticity at the origin; that is,
assume $\int_0^T \lvert \omega(t,\mathbf{0})\rvert \, dt = \infty$. Then the pressure function $P$ given by \eqref{pressuregeneral}
must satisfy $\limsup_{t\to T} (T-t)^2 P_{rr}(t,\mathbf{0}) \ge \tfrac{1}{4}$.
\end{theorem}

Superficially it seems easier to achieve blowup in the case where the angular velocity $b$ is odd than when it is even, especially in
light of Theorem \ref{alpharswirlaxistheorem}. However the situation changes dramatically on the boundary $r=1$ of the cylinder. In that case
numerical evidence~\cite{houluo} shows that the solution of \eqref{eulergeneral} ends in finite time due to blowup of \emph{odd} data, yet our
result below shows that blowup on the boundary is actually easier in the even case.

\begin{theorem}\label{blowupiseasy}
Let $\mathfrak{b}_1 = b_0(1,0)$ and $\mathfrak{b}_2 = (b_0)_r(1,0)$, and assume that $2\mathfrak{b}_1(\mathfrak{b}_1+\mathfrak{b}_2)=-c^2<0$ and $\mathfrak{a}=a_r(0,1,0)<0$. Then
either of the following is sufficient to ensure blowup:
\begin{itemize}
\item $-k^2 \le P_{rr}(t,1,0) + 3\mathfrak{b}_1^2 \le 0$, where $k$ is the positive solution of $k(k+\mathfrak{a})=c^2\ln{2}$, for all time; or
\item $P_{rr}(t,1,0) + 3\mathfrak{b}_1^2 \ge 0$ for all time.
\end{itemize}
\end{theorem}

It is also possible to prove blowup based on assumptions about the time derivative of pressure;
although such assumptions are difficult to justify physically, it is interesting that relatively
mild assumptions can have rather strong consequences. For example we show the following theorem.

\begin{theorem}\label{forcederivative}
Suppose $c_0$ is odd in $z$ and $a_0$ is even in $z$, with $b_0$ either even or odd in $z$.
Assume that $P_{zz}(0,1,0)<0$, that $c_z(0,1,0)<0$, and that $c_z(0,1,0)^2 + P_{zz}(0,1,0)>0$.
If $P_{zz}(t,1,0)$ is increasing for all time, then the solution of \eqref{eulergeneral} ends in finite time.
The same is true if $(t,1,0)$ is replaced everywhere with $(t,0,0)$.
\end{theorem}

The motivation and context for the results described above comes primarily from the geometric point of view,
in particular in terms of the local geometric behavior near a fixed particle trajectory.
Generally, particle trajectories are given in terms of the Lagrangian flow $\eta\colon [0,T)\times M\to M$ by
\begin{equation}\label{flowequation}
\eta_t(t,x) = U\big(t, \eta(t,x)\big), \qquad \eta(0,x) = x.
\end{equation}
The divergence-free condition $\diver{U}=0$ 
implies that the Jacobian is preserved:
\begin{equation}\label{Jacobian}
\Jac(\eta) = \det{(D\eta)} \equiv 1.
\end{equation}
The vorticity transport equation \eqref{vorticity} becomes, in Lagrangian form, the equation
$\omega\big(t,\eta(t,x)\big) = D\eta(t,x) \omega_0(x)$,
from which it is clear by \eqref{BKM} that it is sufficient to understand when $D\eta$ approaches
infinity (or equivalently zero). Combining \eqref{eulergeneral} and
\eqref{flowequation} we obtain the Lagrangian form of the Euler equation,
\begin{equation}\label{lagrangianfluid}
\frac{D}{\partial t}\eta_t(t,x) = -\grad P\big(t, \eta(t,x)\big),
\end{equation}
in terms of the covariant derivative. Differentiating \eqref{lagrangianfluid} in space, we see that $D\eta$ satisfies the linear ODE
\begin{equation}\label{jacobiequation}
\frac{D}{\partial t}\frac{D}{\partial t}D\eta(t,x) = -\nabla^2 P\big(t,\eta(t,x)\big) \, D\eta(t,x),
\end{equation}
(The equation in this form is due to Ohkitani~\cite{ohkitani}.)

We now describe the Riemannian geometric interpretation of equation \eqref{jacobiequation} (based on the work
of Arnold~\cite{arnold}). 
The configuration space is the space of volume-preserving diffeomorphisms (volumorphisms) of $M$, a group under composition,
with right-invariant $L^2$ Riemannian metric given by the kinetic energy; geodesics in this metric are given by
curves $\eta$ satisfying the flow equation \eqref{flowequation} with velocity field $U$ solving the Euler equation
\eqref{eulergeneral}.
The solution operator $U_0\mapsto \eta(1)$ takes $H^s$ velocity fields to $H^s$ volumorphisms for $s>\tfrac{5}{2}$,
and this map is called the Riemannian exponential map. Ebin-Marsden~\cite{ebinmarsden} showed that this map is $C^{\infty}$, and Ebin-Misio{\l}ek-Preston~\cite{ebinmisiolekpreston} showed that the map is not Fredholm, which means that its
differential does not have finite-dimensional kernel or cokernel in general.

The failure of Fredholmness is related to the possibility of infinite clusters of conjugate points, and in fact we
see that accumulation of vorticity up to a blowup time $T$ generically implies the existence of pairs $t_n, t_{n+1}$ such
that the geodesic $\eta$ fails to minimize on $[t_n,t_{n+1}]$ for a sequence $t_n\nearrow T$. This should be viewed as
a ``positive curvature blowup,'' since it cannot happen if the Riemannian sectional curvature is bounded above. In case
this does not happen, we get fairly precise information on the growth of eigenvalues of the stretching matrix~\cite{prestonblowup}.


\begin{theorem}\label{conjugateblowupyes}
Suppose $c_0$ is odd and $b_0$ is even, with $\int_0^T \lvert \omega(t,r_0,0)\rvert \, dt=\infty$ for either $r_0=0$ or $r_0=1$.
Then there is an infinite increasing sequence $t_n\nearrow T$ with $\eta(t_n)$ conjugate to $\eta(t_{n+1})$ if either
\begin{itemize}
\item $r_0=0$ or
\item $r_0=1$ and $4b_0(1,0)^2 + 4b_0(1,0) (b_0)_r(1,0)>\lvert \alpha_{rt}(T,1,0)\rvert^2$,
\end{itemize}
where $\alpha$ is the radial component of the diffeomorphism $\eta$. 
In either case the supremum of sectional curvatures $\sup_V K(U,V)$ approaches positive infinity in the sense that 
$$\int_0^T \sup_{V\in T_{\id}\Diffmu(M)} K(U(t),v) \, dt = \infty.$$
\end{theorem}




\begin{theorem}\label{conjugateblowupno}
Suppose $c_0$ and $b_0$ are both odd, with $\int_0^T \lvert \omega(t,r_0,0)\rvert \, dt=\infty$
for either $r_0=0$ or $r_0=1$. Then we cannot find curve-shortening variations on time intervals $[a,b]\subset [0,T]$ that are
supported in arbitrarily small neighborhoods of the point $r=r_0$, $z=0$. In addition
the pressure Laplacian and the sectional curvature must both approach negative infinity in the
sense that
\begin{equation}\label{infimumlaplacian}
\int_0^T \inf_{x\in M} \Delta P(t,x) \, dt = \int_0^T \inf_{V\in T_{\id}\Diffmu(M)} K(U(t),V) \, dt = -\infty.
\end{equation}
\end{theorem}

We thus see that there is a basic dichotomy between positive-curvature blowup scenarios, corresponding
intuitively to convex pressure and dominance of rotation, and negative-curvature scenarios, corresponding
to concave pressure and dominance of stretching.

We now describe the plan of the paper. First in Section \ref{backgroundsection}, we present the axisymmetric
3D Euler equation \eqref{eulergeneral} in Lagrangian form.
From here we compute the equation \eqref{jacobiequation} in components and specialize to the equations on the
symmetry axis and on the boundary, where they simplify drastically. Then we discuss further simplifications that
arise if the data is assumed to have additional reflection symmetries, in order to obtain a system of ODEs, and relate the resulting equations to well-known dynamical systems such as the Ermakov-Pinney equation.  In Section \ref{blowupcriteria} we collect conditions on the pressure Hessian leading to blowup at the fixed points and establish Theorems \ref{alpharswirlaxistheorem}--\ref{forcederivative}. Then in Section \ref{globalgeometry} we prove Theorems \ref{conjugateblowupyes}--\ref{conjugateblowupno}, thus giving a geometric
interpretation of our blowup results both in terms of the local geometry along trajectories and in
terms of the Riemannian geometry of the volume-preserving diffeomorphism group. Lastly, in Section \ref{outlook} we discuss future work by relating the axisymmetric 3D Euler
equation to two lower-dimensional models we believe have similar geometric structure: In one dimension, an equation proposed by Wunsch~\cite{wunsch} as a geometric model of 3D Euler,
based on an idea of Constantin-Lax-Majda~\cite{constantinlaxmajda} and developed by De Gregorio~\cite{degregorio}, and in two dimensions, the surface quasigeostrophic
equation (SQG), which is used as a lower dimensional analogue of the 3D Euler equation.

\section{Background}\label{backgroundsection}

Let us first establish our notation. We denote the velocity field by $U(t,r,z)$, and write it in components as
$U = a e_r + be_{\theta} + c e_z$, where
$e_r$, $e_{\theta}$, and $e_z$ are
the usual cylindrical unit vector fields and the coefficients $a$, $b$, $c$ all depend on $(t,r,z)$.
We denote the Lagrangian flow $\eta\colon [0,T)\times M\to M$ by
\begin{equation}\label{flowcompdef}
\eta(t,r,\theta,z) = \big( \alpha(t,r,z), \theta + \beta(t,r,z), \gamma(t,r,z)\big),
\end{equation}
with initial conditions $\alpha(0,r,z) = r$, $\beta(0,r,z)=0$, and $\gamma(0,r,z)=z$. 
The Lagrangian flow equation \eqref{flowequation} takes the form
\begin{equation}\label{flowcomponents}
\alpha_t = a(t,\alpha,\gamma), \qquad \beta_t = \frac{\alpha}{r} b(t,\alpha,\gamma), \qquad \gamma_t = c(t,\alpha,\gamma).
\end{equation}


The following conservation law for angular vorticity is one of the most important for us.

\begin{proposition}\label{Bsolutionprop}
For any axisymmetric solution $U = ae_r + be_{\theta}+ce_z$ of \eqref{eulergeneral} with flow
components \eqref{flowcomponents}, we have
\begin{equation}\label{Bsolution}
\alpha(t,r,z) b\big(t,\alpha(t,r,z), \gamma(t,r,z)\big) = rb_0(r,z).
\end{equation}
\end{proposition}

\begin{proof}
This comes from writing the angular component of \eqref{eulergeneral} as $b_t + ab_r + cb_z + ab/r = 0$.
Composing with the flow and using \eqref{flowcomponents}, we get that $\partial_t(\alpha \beta_t)=0$.
Integrating and using $\alpha|_{t=0}=r$, we get \eqref{Bsolution}.
\end{proof}


%

Incorporating the conservation law of Proposition \ref{Bsolutionprop}, we get convenient forms for
the $\alpha$ and $\gamma$ equations.

\begin{proposition}\label{alphagammaprop}
The components $\alpha$ and $\gamma$ of the flow \eqref{flowcompdef} satisfy
\begin{equation}\label{alphagammaeqs}
\alpha_{tt} - \frac{r^2 b_0^2}{\alpha^3} = -P_r(t,\alpha,\gamma) \quad \text{and}\quad \gamma_{tt} = -P_z(t,\alpha,\gamma).
\end{equation}
while the Jacobian determinant condition \eqref{Jacobian} is
\begin{equation}\label{jacobianexplicit}
\alpha(\alpha_r \gamma_z - \alpha_z \gamma_r) \equiv r.
\end{equation}
\end{proposition}

\begin{proof}
The formulas \eqref{alphagammaeqs} are proved by writing the equation \eqref{eulergeneral} in components as
$a_t + aa_r + ca_z - b^2/r = -P_r$ and $c_t + ac_r + cc_z = -P_z$, and using equations \eqref{flowcomponents}
to get second-order equations for $\eta$, then plugging in \eqref{Bsolution}. The formula \eqref{jacobianexplicit}
is straightforward.
\end{proof}



There are effectively two boundaries on the cylinder: at $r=0$ along the axis, and at $r=1$. They are slightly
different since on the axis the constraints are determined by the requirement of smoothness and rotational
invariance, while on the boundary they are determined by the no-flow condition, although they are ultimately
similar since $\alpha$ is fixed in either case: $\alpha(t,0,z) = 0$ and $\alpha(t,1,z)=1$.



\begin{lemma}\label{incompressibilitylemma}
At the fixed radii $r=0$ and $r=1$,
the incompressibility condition \eqref{jacobianexplicit} takes the form
\begin{equation}\label{jacobianfixedpoint}
\alpha_r(t,0,z)^2 \gamma_z(t,0,z) = 1 \qquad \text{and}\qquad \alpha_r(t,1,z) \gamma_z(t,1,z) = 1.
\end{equation}
\end{lemma}

\begin{proof}
This follows from the fact that $\alpha(t,r_0,z)=\epsilon$ for $r_0 = 0$ or $r_0=1$, so
that $\alpha_z(t,\epsilon,z)=0$ and \eqref{jacobianexplicit} becomes $(\alpha/r)\alpha_r \gamma_z=1$.
The term $\alpha/r$ becomes $\alpha_r$ at $r_0=0$ and becomes unity at $r_0=1$.
\end{proof}


We obtain fixed points for the flow (for any $\theta$) at $z=0$ if we assume that the flow is symmetric
about the plane $z=0$, in the sense that $c_0$ is odd in $z$: that is, $c_0(r,-z) = -c_0(r,z)$. The following
result is well-known in the case when $b_0$ is odd, but the case $b_0$ even has not been as widely studied.

\begin{proposition}\label{oddeven}
Let $U_0$ be a divergence-free velocity field with components $U_0 = a_0 e_r + b_0 e_{\theta} + c_0 e_z$ on the solid
torus $M = D\times S^1$, and assume that $a_0$ is even, $c_0$ is odd, and $b_0$ is either even or odd in the $z$ variable
(viewed as an element of $[-\pi,\pi]$). Then the same symmetries hold for $U$ for all time. In particular the origin
$r=0$ and $z=0$ is a fixed point of the flow, as are all points on the circle $z=0$ and $r=1$.
\end{proposition}

\begin{proof}
Proposition \ref{Bsolutionprop} implies that if $b_0^2$ is even, then so is $b^2$ for all time.
It is then easy to check that the equations \eqref{alphagammaeqs} are preserved under the assumed symmetries, since
writing out equation \eqref{pressuregeneral} in components as
\begin{equation}\label{pressurelapcomps}
\tfrac{1}{r} \partial_r (r P_r) + P_{zz} = -a^2/r^2 - a_r^2 - 2a_z c_r - c_z^2 + 2bb_r/r
\end{equation}
shows that $P$ will be even in $z$. We then have $c(t,r,0)=0$ for all $t$ and $r$, so that by \eqref{flowcomponents}
we have $\gamma(t,r,0) = 0$ for all $t$ and $r$. Since we also have $\alpha(t,0,z)=0$ and $\alpha(t,1,z)=1$,
we conclude that $z=0$ and $r=0$ or $r=1$ are all fixed points.
\end{proof}

Note that the origin is a genuine fixed point, while the circle $r=1$ and $z=0$ is fixed by the flow but may rotate
since $b$ is not necessarily zero there (if it is not odd).

Now differentiating the equations \eqref{alphagammaeqs} spatially, we obtain linear ODEs
for the components of the matrix $D\eta$ at these fixed points.

\begin{proposition}\label{pointwiselinear}
Suppose $U_0$ is a vector field satisfying the assumptions of Proposition \ref{oddeven}.

Then on the axis $f_0(t)=\alpha_r(t,0,0)$ and $g_0(t) = \gamma_z(t,0,0)$ satisfy the ODEs
\begin{equation}\label{axisBeven}
f_0''(t) - \frac{\mathfrak{b}_0^2}{f_0(t)^3} = -P_{rr}(t,0,0) f_0(t) \qquad \text{and}\qquad g_0''(t) = -P_{zz}(t,0,0) g_0(t)
\end{equation}
with the constraints $f_0(t)^2 g_0(t) = 1$ and
\begin{equation}
\label{pressureaxis}
2P_{rr}(t,0,0) + P_{zz}(t,0,0) = -6f_0'(t)^2/f(t)^2 + 2\mathfrak{b}_0^2/f(t)^4,
\end{equation}
where $\mathfrak{b}_0 = \partial_rb_0(0,0)$. And on the boundary, $f_1(t) = \alpha_r(t,1,0)$ and $g_1(t) = \gamma_z(t,1,0)$ satisfy the ODEs
\begin{equation}\label{boundaryBeven}
f_1''(t) - 2\mathfrak{b}_1 (\mathfrak{b}_1 + \mathfrak{b}_2) + 3\mathfrak{b}_1^2 f_1(t) = -P_{rr}(t,1,0) f_1(t)\,\quad \text{and}\,\quad
g_1''(t) = -P_{zz}(t,1,0) g_1(t),
\end{equation}
with the constraints $f_1(t)g_1(t)=1$ and
$$P_{rr}(t,1,0) + P_{zz}(t,1,0) = -3\mathfrak{b}_1^2 + \frac{2\mathfrak{b}_1(\mathfrak{b}_1+\mathfrak{b}_2)}{f_1(t)} - 2f_1'(t)^2/f_1(t)^2,$$
where $\mathfrak{b}_1 = b_0(1,0)$ and $\mathfrak{b}_2 = (b_0)_r(1,0)$.
\end{proposition}

The equations \eqref{axisBeven} and \eqref{boundaryBeven} are essentially like the Jacobi equations in
a three-dimensional manifold along a geodesic, where the Hessian of the pressure acts as the
effective sectional curvature; this analogy is made more precise in \cite{prestonfirst, prestonwkb}.
The equation \eqref{axisBeven} is essentially the  Ermakov-Pinney equation \cite{ermakovpinneyreview}
which describes a planar harmonic oscillator with central returning force. We have blowup if any of
the terms $f_0(t)$ or $g_0(t)$ become zero in finite time (since the constraints mean that if one is zero, the
other approaches infinity). Our main technique will thus be the comparison theory for linear ODEs.

\section{Blowup criteria}\label{blowupcriteria}

Here we collect some conditions on the pressure Hessian which would ensure blowup at the fixed points of the equations.
They are all based on the fact that some component of $D\eta$ must approach zero in order for the other component to
approach infinity, and it is of course easier to ensure that the solution of a linear differential equation approaches
zero than infinity since we need not require the coefficients to approach infinity. Geometrically we are asking that the
Riemannian curvature is positive enough that we get conjugate points, since $D\eta$ generates the Jacobi fields and we
want to see them vanishing.

\subsection{Blowup criteria on the axis}\label{blowupaxis}

The following result generalizes the case $\epsilon=0$ with $P_{rr}(t,0,0)\ge 0$ considered by Chae~\cite{chae1}; here
we consider blowup in terms of a linear ODE for the component of $D\eta$ rather than a Riccati equation for the velocity
component.

\begin{theorem}\label{gammazaxistheorem}
Suppose $c_0$ is odd, $a_0$ is even, and $b_0$ is odd or even in $z$. Then any of the following is
sufficient to ensure blowup in finite time, for either $r_0=0$ or $r_0=1$.
\begin{itemize}
\item $P_{rr}(t,r_0,0) \ge 0$ for all $t$, and initially $a_r(0,r_0,0)<0$ and $b_0(0,r_0,0)=0$;
\item $P_{zz}(t,r_0,0) \ge 0$ for all $t$, and initially $c_z(0,r_0,0)<0$.
\end{itemize}
\end{theorem}

\begin{proof}
All four cases are the same: if $g(t) = \alpha_r(t,r_0, 0)$ or $g(t) = \gamma_z(t,r_0,0)$,
the assumptions are equivalent to $g(0)=1$, $g'(0)<0$, and $g''(t)/g(t)\le 0$.
Hence comparing with $g_0(t) = 1 + g'(0)t$ we have $g(T)=0$ no later than $T = -1/g'(0)$.
\end{proof}

Note that while Chae's result~\cite{chae1} only
requires $P_{rr}$ to be positive along the axis (not necessarily unbounded), equation \eqref{pressureaxis}
clearly requires $P_{zz}$ to blow up to negative infinity in this case, corresponding to the fact that boundedness of the
pressure Hessian is also sufficient to prevent blowup~\cite{chae2007}.

In the case where $(b_0)_r(0,0)\ne 0$, the situation becomes more
interesting. It is impossible in this case for $\alpha_r$ to approach zero in finite time if
$P_{rr}(t,0,0)$ is bounded, and in fact using fairly standard Sturm-Liouville comparison theory~\cite{swanson}, we can get a blowup rate for $P_{rr}(t,0,0)$ if $\alpha_r$ does approach zero and satisfies a localized BKM-type criterion in the form
$$\int_0^T \lvert \omega(t,0,0)\rvert \, dt = 2 \lvert b_0\rvert \int_0^T \frac{dt}{\alpha_r(t,0,0)^2} =  \infty;$$
see Theorem \ref{conjugatepoints}.


\noindent
\textbf{Proof of Theorem \ref{alpharswirlaxistheorem}.}
Consider the equation $\ddot{\rho}(t) - b_0^2/\rho(t)^3 = -F(t) \rho(t)$, where $\rho(t) = \alpha_r(t,0,0)$ and $F(t) = P_{rr}(t,0,0)$.
As pointed out by Eliezer and Gray~\cite{eliezergray}, this is the
equation for the radial coordinate $\rho(t)$ for a planar central force system $\ddot{x}(t) = -F(t)x(t)$, $\ddot{y}(t) = -F(t)y(t)$,
where $\rho(t)^2 = x(t)^2 + y(t)^2$ and $x(0)\dot{y}(0) - \dot{x}(0)y(0) = b_0$. The angular coordinate $\theta(t)$ then satisfies
$\dot{\theta}(t) = b_0/\rho(t)^2$, and the condition $\int_0^T dt/\rho(t)^2 = \infty$ means exactly that $\theta(t)$ winds around the
origin infinitely many times as $t\to T$. Hence $x(t)$ and $y(t)$ have infinitely many zeroes as $t\to T$, and thus \emph{every}
solution of $\ddot{g}(t) = -F(t)g(t)$ has infinitely many zeroes on $(0,T)$.

We now change variables so that the blowup time is sent to infinity: set $s = -\ln{(T-t)}$, set $g(t) = (T-t)^{1/2} j(s)$,
and $F(t) = H(s)/(T-t)^2$. Then the equation $\ddot{g}(t) = -F(t)g(t)$ becomes
$$ j''(s) = \left(\tfrac{1}{4} - H(s)\right) j(s),$$
and we must have $\limsup_{s\to \infty} H(s)-\frac{1}{4} \ge 0$ for solutions of this equation to have
infinitely many zeroes. Translating back in terms of $P_{rr}$ we obtain
\begin{equation}
\label{ermakovpressureblowup}
\limsup_{t\to T} (T-t)^2 P_{rr}(t,\mathbf{0}) \ge \tfrac{1}{4}.
\end{equation}
$\square$

With more assumptions one can obtain more precise criteria, using the methods presented e.g., in Swanson~\cite{swanson},
but for our purposes \eqref{ermakovpressureblowup} already makes clear how tightly the pressure Hessian is constrained
in a typical blowup scenario at a fixed point with nonzero vorticity.

Theorem \ref{blowupiseasy} shows that we can get the same sort of blowup as in
Theorem \ref{gammazaxistheorem} even if the pressure Hessian is negative, as long
as there is enough rotation to compensate.

\noindent
\textbf{Proof of Theorem \ref{blowupiseasy}.}
First we work under the condition $-k^2 \le P_{rr}(t,1,0) + 3\mathfrak{b}_1^2 \le 0$.
Write $f_1(t) = \alpha_r(t,1,0)$ and $F(t) = P_{rr}(t,1,0)+3\mathfrak{b}_1^2$; then by equation \eqref{boundaryBeven}, $f_1$ satisfies the equation
\begin{equation}\label{gequation}
f_1''(t) = -c^2 - F(t) f_1(t), \qquad f_1(0)=1,\,  f_1'(0)=-\mathfrak{a},
\end{equation}
where $0\le -F(t)\le k^2$ and $\mathfrak{a}>0$. Consider the solution $y_1(t)$ of the related problem
$$ y_1''(t) = -F(t) y_1(t), \qquad y_1(0)=1,\, y_1'(0)=0.$$
By the usual Sturm comparison theorem~\cite{swanson}, we have
\begin{equation}\label{y1inequality}
1\le y_1(t) \le \cosh{kt}\quad \text{ for all $t$.}
\end{equation}
Using the reduction of order trick, the solution of $y_2''(t) = -F(t)y_2(t)$ with $y_2(0)=0$ and $y_2'(0)=1$
is given by $y_2(t) = y_1(t) \int_0^t ds/y_1(s)^2$. Using variation of parameters, we may write the solution
$f_1(t)$ of \eqref{gequation} as
\begin{align*}
f_1(t) &= y_1(t) - \mathfrak{a} y_2(t) + c^2y_1(t) \int_0^t y_2(s)\,ds - c^2 y_2(t) \int_0^t y_1(s)\,ds \\
&= y_1(t)\left[ 1 - \mathfrak{a}\int_0^t \frac{ds}{y_1(s)^2} - c^2 \int_0^t \int_s^t \frac{y_1(s)}{y_1(\tau)^2} \, d\tau \, ds\right].
\end{align*}
The inequality \eqref{y1inequality} now implies that
\begin{align*}
\frac{f_1(t)}{y_1(t)} &\le 1 - \mathfrak{a}\int_0^t \frac{ds}{\cosh^2{ks}} - c^2 \int_0^t \int_s^t \frac{d\tau}{\cosh^2{k\tau}} \, ds \\
&= 1 - \frac{\mathfrak{a}}{k} \tanh{kt} - \frac{c^2\ln{2}}{k^2} + \frac{c^2}{k^2} \ln{\big(1+e^{-2kt}\big)} + \frac{2c^2t}{k(1+e^{2kt})},
\end{align*}
and thus we have
$$ \lim_{t\to\infty} \frac{f_1(t)}{y_1(t)} \le \frac{k^2 - \mathfrak{a}k - c^2 \ln{2}}{k^2}.$$
We conclude that $f_1(t)$ is eventually negative.

The second assumption that $P_{rr}(t,1,0) + 3\mathfrak{b}_1^2 \ge 0$ is much easier: in this case we just
have $\alpha_{rtt}(t,1,0) \le -c^2$, and since $\alpha_r(0,1,0)=1$ and $\alpha_{tr}(0,1,0)<0$, we obviously have $\alpha_r(t,1,0)$ reaching zero in finite time.\hfill$\square$


In the previous theorems, we showed how some assumptions on the sign of certain components
of the pressure Hessian could cause blowup on
either the axis or on the boundary. We now change our perspective somewhat and look
at the consequences of assumptions on the time derivative of components of the pressure
Hessian.


\noindent
\textbf{Proof of Theorem \ref{forcederivative}.}
From Proposition \ref{pointwiselinear} we see that the simplest case is the $\gamma_z$ component, since in all cases ($b$ even or odd, or whether we work on the boundary
or the axis), $\gamma_z$ satisfies the
simplest equation $\gamma_{ttz} = -P_{zz}(t) \gamma_z$. Hence we need only consider the equation
\begin{equation}\label{basicgequation}
g''(t) = -Q(t) g(t), \qquad g(0)=1, \quad g'(0)=-\mathfrak{a},
\end{equation}
where $\mathfrak{a}>0$, $Q(0)<0$, and $\nu^2 = \mathfrak{a}^2 + Q(0)>0$, and we assume $Q'(t)\ge 0$ for all $t$.

Multiplying \eqref{basicgequation} by $g'$ and integrating, we obtain
$$ g'(t)^2 + Q(t)g(t)^2 = \nu^2 + \int_0^t Q'(\tau) g(\tau)^2 \, d\tau.$$
Rearranging this now gives
$$ g'(t)^2 - \nu^2 = -Q(0) g(t)^2 + \int_0^t Q'(\tau) \big[g(\tau)^2-g(t)^2\big] \, d\tau.$$
Since $g(0)=1$ and $g'(0)<0$, we know $g$ is decreasing and positive on some time interval $[0,T_0]$. On this interval we have $g(\tau)^2-g(t)^2\ge 0$ whenever $\tau\le t\le T_0$. We conclude that on this interval, $g'(t)^2 - \nu^2 \ge 0$, which implies that $g'(t)\le -\nu$ as long as $g(t)$ is decreasing and positive. Since $g'$ is continuous, we conclude that $g$ must reach zero before it changes direction, and furthermore we have $g(t) \le 1-\nu t$ so that the time $T$ of the first zero is no larger than $1/\nu$.\hfill $\square$

Of course the same theorem applies with $\alpha_r$ replacing $\gamma_z$, in case $b_0$ is assumed to be odd, since then
the equations \eqref{axisBeven}--\eqref{boundaryBeven} all reduce to the same equation $f''(t) = -Q(t) f(t)$.
Similar theorems could be proved using the more complicated equations for $\alpha_r$ arising from \eqref{axisBeven} and \eqref{boundaryBeven} in case $b_0$ is even, but we will leave these aside for now.

\section{Global geometry of the Euler equation}\label{globalgeometry}

In the previous Section we made a variety of assumptions on the local behavior of the fluid which could lead to blowup; here we would like to tie this local picture into the global behavior of the equation (especially as related to the Riemannian geometry of the volume-preserving diffeomorphism group) and the global behavior of the pressure function.

\subsection{Conjugate points and blowup in axisymmetric fluids}

Viewed as a Riemannian manifold, the group of volumorphisms $\Diffmu(M^3)$ with Riemannian metric
$\llangle u,u\rrangle = \int_M \langle u,u\rangle \, d\mu$ has geodesics satisfying
\eqref{lagrangianfluid} (which is equivalent to \eqref{flowequation} and \eqref{eulergeneral}), as pointed out by Arnold~\cite{arnold}. Its sectional curvature describes small Lagrangian perturbations, and is given for divergence-free velocity fields $U$ and $V$, using the Gauss-Codazzi formula~\cite{misiolek}, by
\begin{equation}\label{gausscodazzi}
\llangle R(U,V)V,U\rrangle_{L^2} = \int_M \nabla^2 P(V,V)\, d\mu - \int_M \lvert \grad Q\rvert^2 \, d\mu,
\end{equation}
where $P$ is the pressure of the velocity field $U$ and $Q$ solves the Neumann problem
$\Laplacian Q = -\diver{(U\cdot \nabla V)}$ with $\langle \grad Q+U\cdot\nabla V, e_r\rangle_{r=1} = 0$.
In the sense that the Hessian of the pressure always plays the role of the ``effective'' curvature along local Jacobi fields,
as mentioned after Proposition \ref{pointwiselinear}, the formula \eqref{gausscodazzi} shows that the actual Riemannian
curvature is always less than this ``effective'' curvature term.

The following theorem shows that conjugate points on the volumorphism group can be found by a local criterion along a single particle
path, which demonstrates that it is easy to shorten a geodesic curve by performing rotational perturbations near a point. This is a
purely three-dimensional result: Fredholmness of the 2D Riemannian exponential map~\cite{ebinmisiolekpreston} implies there is no similar result in two dimensions.

\begin{theorem}\label{indexformtheorem}\cite{prestonfirst}
Suppose $M$ is a three-dimensional manifold and $x$ is in the interior of $M$, and $\eta$ is a Lagrangian solution of the ideal
Euler equations with $\eta_t(0,x) = u_0(x)$. Let $\Lambda(t,x) = D\eta(t,x)^{\transpose} D\eta(t,x)$, and let $\omega_0(x)=\curl U_0(x)$ denote the initial
vorticity at $x$.
If there is a vector field $v(t)$ along the Lagrangian trajectory $t\mapsto \eta(t,x)$ with $v(t_1)=v(t_2)=0$ such that
\begin{equation}\label{localizedindexform}
I(v,v) = \int_{t_1}^{t_2} \langle \Lambda(t,x) \dot{v}(t), \dot{v}(t)\rangle + \langle \omega_0(x) \times v(t), \dot{v}(t)\rangle \, dt < 0,
\end{equation}
then $\eta(t_1)$ is conjugate to $\eta(\tau)$ for some $\tau<t_2$; in particular the geodesic $\eta$ is not minimizing
on $[t_1,t_2]$.
\end{theorem}


In general for axisymmetric flows in the orthonormal basis $\{e_r,e_{\theta},e_z\}$ we have
$$
D\eta(t,x) = \left(\begin{matrix}
\alpha_r & 0 & \alpha_z \\
\alpha \beta_r & \alpha/r & \alpha \beta_z \\
\gamma_r & 0 & \gamma_z
\end{matrix}\right),
$$
and since \eqref{flowequation} for $\beta$ reduces to $\frac{\partial \beta}{\partial t}(t,r,z) = b_0(r,z)$ by the conservation law \eqref{Bsolution}, we have $\beta_r(t,r,z) = t(b_0)_r(r,z)$ and $\beta_z(t,r,z) = t(b_0)_z(r,z)$. Thus we have
\begin{equation}\label{Detaexplicit}
D\eta(t,0,0) = \left(\begin{matrix}
\alpha_r & 0 & 0 \\
0 & \alpha_r & 0 \\
0 & 0 & \gamma_z\end{matrix}\right),
\qquad\qquad D\eta(t,1,0) = \left(\begin{matrix}
\alpha_r & 0 & 0 \\
t(b_0)_r(1,0) & 1 & t(b_0)_z(1,0) \\
0 & 0 & \gamma_z\end{matrix}\right).
\end{equation}


We now analyze the index form  on the axis and on the boundary.
In all cases we assume the localized Beale-Kato-Majda criterion \eqref{localizedBKM}:
at the fixed point $z=0$ with either $r_0=0$ or $r_0=1$, we have
$$\int_0^T \lvert \omega\big(t,\eta(t,r_0, 0)\big)\rvert \, dt = \infty \quad \text{for\, $r_0=0$\, or\, $r_0=1$}.$$
Note that if $b_0$ is odd, then the vorticity is identically zero at the origin, and the localized BKM criterion cannot be satisfied there.
We first consider the assumptions necessary to get an infinite sequence of conjugate pairs (and thus sectional curvature increasing to
positive infinity).

\noindent
\textbf{Proof of Theorem \ref{conjugateblowupyes}.}
First consider the situation at $r=0$.
It is sufficient to show that for any $t_1>0$ there is a $t_2>t_1$ such that the index form $I(v,v)$ in \eqref{localizedindexform} can be
made negative for some $v$ vanishing at both $t_1$ and $t_2$. The initial vorticity is given by $\omega_0 = 2\mathfrak{b}_0 \, e_z$, and since it is stretched by
$\omega(t,0,0) = \gamma_z(t,0,0) \omega_0(0,0) = 2\mathfrak{b}_0/\alpha_r(t,0,0)^2$, our assumption yields
$\int_0^T dt/\alpha_r(t,0,0)^2 = \infty$. Equation \eqref{Detaexplicit} yields
$$ \Lambda(t,0,0) = D\eta(t,0,0)^{\dagger}D\eta(t,0,0) = \left(\begin{smallmatrix}
\alpha_r^2 & 0 & 0 \\
0 & \alpha_r^2 & 0 \\
0 & 0 & \gamma_z^2\end{smallmatrix}\right),
$$
so that the index form \eqref{localizedindexform} becomes, for $v(t) = f(t) e_r + g(t) e_{\theta} + h(t) e_z$,
$$ I(v,v) = \int_{t_1}^{t_2} \alpha_r^2 \dot{f}^2 + \alpha_r^2 \dot{g}^2 + \gamma_z^2 \dot{h}^2 + \mathfrak{b}_0 (f\dot{g} - g\dot{f}) \, dt.$$

Set $h\equiv 0$, integrate by parts using $v(t_1)=v(t_2)=0$, and complete the square to obtain
$$ I(v,v) = \int_{t_1}^{t_2} \big( \alpha_r \dot{g} + \tfrac{\mathfrak{b}_0}{\alpha_r} f\big)^2 + \alpha_r^2 \dot{f}^2 - \tfrac{\mathfrak{b}_0^2}{\alpha_r^2} f^2 \, dt.$$
We choose $\dot{g} = k - \mathfrak{b}_0 f/\alpha_r^2$, where $k$ is chosen so that $\int_{t_1}^{t_2} \dot{g} \, dt = 0$, and obtain
\begin{equation}\label{indexsimplified1}
I(v,v) = \frac{\mathfrak{b}_0^2}{(t_2-t_1)^2} \int_{t_1}^{t_2} \alpha_r^2 \,dt \big( \textstyle{\int_{t_1}^{t_2} \tfrac{f\,dt}{\alpha_r^2}}\big)^2 + \displaystyle \int_{t_1}^{t_2} \alpha_r^2 \dot{f}^2 - \tfrac{\mathfrak{b}_0^2}{\alpha_r^2} f^2 \, dt.
\end{equation}
For the latter integral, we rescale our time variable by $s(t)=\int_0^t d\tau/\alpha_r(\tau)^2$, and obtain
\begin{equation}\label{dominantindexaxis}
\int_{t_1}^{t_2} \alpha_r^2 \dot{f}^2 - \tfrac{\mathfrak{b}_0^2}{\alpha_r^2} f^2 \, dt = \int_{s_1}^{s_2} f'(s)^2 - \mathfrak{b}_0^2 f(s)^2 \, ds.
\end{equation}
By assumption we have $s(t)\to \infty$ as $t\to T$, so that we can certainly choose $t$ large enough so that with $f(s) = \sin{2(s-s_1)\pi/\mathfrak{b}_0}$,
the integral in \eqref{dominantindexaxis} is negative.
With this choice, the first integral in \eqref{indexsimplified1} vanishes, and we get $I(v,v)<0$.

Next we consider the situation at $r=1$.
We then have
$$ \Lambda(t,1,0) = \left(\begin{matrix}
\alpha_r^2 + \mathfrak{b}_2^2t^2  & \mathfrak{b}_2t & 0 \\
\mathfrak{b}_2t & 1 & 0 \\
0 & 0 & \gamma_z^2
\end{matrix}\right),$$
and the initial vorticity is $\omega_0(1,0) = (\mathfrak{b}_1+\mathfrak{b}_2) e_z$.
For $v(t) = f(t)e_r + g(t) e_{\theta} + h(t) e_z$,
the index form \eqref{localizedindexform} becomes
$$ I(v,v) = \int_{t_1}^{t_2} (\alpha_r^2 + \mathfrak{b}_2^2t^2) \dot{f}^2 + 2\mathfrak{b}_2 t \dot{f} \dot{g} + \dot{g}^2 + \gamma_z^2 \dot{h}^2 +
(\mathfrak{b}_1+\mathfrak{b}_2) (f\dot{g} - g\dot{f}) \, dt.$$
Set $h=0$, integrate by parts, and complete the square to obtain
\begin{equation}\label{indexformeven}
I(v,v) = \int_{t_1}^{t_2} \big( \dot{g} + (\mathfrak{b}_1+\mathfrak{b}_2) f + \mathfrak{b}_2 t \dot{f}\big)^2 + \alpha_r^2 \dot{f}^2 - \mathfrak{b}_1(\mathfrak{b}_1+\mathfrak{b}_2) f^2 \, dt.
\end{equation}
We have $\omega(t,1,0) = \tfrac{\mathfrak{b}_1+\mathfrak{b}_2}{\alpha_r(t,1,0)} e_z$ so that the blowup condition is
$\int_0^T dt/\alpha_r(t,1,0) = \infty$. If $q:=\alpha_{rt}(T,1,0)\ne 0$ then the dominant term in
\eqref{indexformeven} looks, for $t_1$ and $t_2$ sufficiently close to $T$, like
$$ \int_{t_1}^{t_2} q^2 (T-t)^2 \dot{f}(t)^2 - \tfrac{\zeta^2}{4} f(t)^2 \, dt$$
for $\zeta^2 = 4\mathfrak{b}_1(\mathfrak{b}_1+\mathfrak{b}_2)$. Minimizers of this integral subject to $f(t_1)=f(t_2)=0$
satisfy the equation $\frac{d}{dt} \big( q^2(T-t)^2 \dot{f}(t)\big) + \tfrac{\zeta^2}{4} f(t) = 0$,
with solutions
\begin{equation}\label{minimizersolutionf}
f(t) = \frac{1}{\sqrt{T-t}} \cos{(\tfrac{\psi}{2} \ln{(T-t)} + \phi)}\, \text{ for }\, \psi = \sqrt{\zeta^2/q^2-1}
\end{equation}
and some constant $\phi$, and all such solutions vanish infinitely many times up to time $T$.

As in the case above when $r=0$, we choose $\dot{g} = k - \mathfrak{b}_1 f - \mathfrak{b}_2 \tfrac{d}{dt} (tf)$ where
$k$ is chosen so that $\int_{t_1}^{t_2} \dot{g}(t) \, dt = 0$, i.e., $k = \tfrac{\mathfrak{b}_1}{t_2-t_1} \int_{t_1}^{t_2} f(t)\,dt$. Then the first
term in \eqref{indexformeven} vanishes if $\int_{t_1}^{t_2} f(t)\, dt = 0$, and we can easily choose functions $f$ of the form \eqref{minimizersolutionf}
with vanishing mean if $t_2$ is close enough to $T$.

The result on the integral curvature follows from rather general principles of comparison theory in ODEs; the following is 
adapted from the proof of Theorem 5.1 in Chapter 11 of Hartman~\cite{hartman}. Consider a Jacobi field $J(t)$ satisfying the 
equation $J''(t) = -R(\dot{\gamma}(t), J(t))\dot{\gamma}(t)$. Let $\phi(t) = \tfrac{1}{2} \lvert J(t)\rvert^2$; then we have 
$$ \phi''(t) = -2K(t) \phi(t) + \lvert J'(t)\rvert^2 \ge -2K(t) \phi(t),$$
assuming that $\dot{\gamma}(t)$ is a unit vector, where $K(t)$ is the sectional curvature in directions $J(t)$ and $\dot{\gamma}(t)$. 
On an interval $[t_n,t_{n+1}]$ where a Jacobi field vanishes, we have 
$$ (t_{n+1}-t_n) \phi(t) \le 2\int_{t_n}^t (t_{n+1}-t) (s-t_n) K(s) \phi(s) \, ds + 2\int_t^{t_{n+1}} (t-t_n) (t_{n+1}-s) K(s) \phi(s) \, ds$$ 
for any $t \in [t_n, t_{n+1}]$. If $t_0$ denotes the location of the maximum of $\phi(t)$ then using $t=t_0$ and overestimating the right side by $\phi(s)\le \phi(t_0)$, then cancelling $\phi(t_0)$, we get
$$ (t_{n+1}-t_n) \le \tfrac{1}{2} (t_{n+1}-t_n)^2 \int_{t_n}^{t_{n+1}} K(s) \, ds.$$
Therefore we have 
$$ \int_0^T K(s) \, ds \ge  \sum_{n=1}^{\infty} \int_{t_n}^{t_{n+1}} K(s) \, ds \ge \sum_{n=1}^{\infty} \frac{2}{t_{n+1}-t_n}.$$
Since the sum $\sum_{n=1}^{\infty} (t_{n+1}-t_n) \le T$ converges, the reciprocal sum must diverge, and we conclude that $\int_0^T K(t) \,dt$ 
is positive infinity. \hfill $\square$

More generally we can try to generalize the preceding computations to \emph{any} Lagrangian path along which the vorticity is
increasing to infinity, but it turns out there are not always infinitely many conjugate pairs along the trajectory. The
alternate possibility is however rare enough that we can get fairly concrete information about the growth of the components
of the stretching matrix, as in the following theorem of the first author.

\begin{theorem}\label{conjugatepoints}\cite{prestonblowup}
Assume that a solution of the 3D Euler equation \eqref{eulergeneral} on a 3D manifold $M$ has a maximal existence time $T<\infty$, and that the following strong form of the Beale-Kato-Majda criterion \eqref{BKM} holds:
\begin{equation}\label{localizedBKM}
\exists x\in M\backslash\partial M \text{ s.t. } \int_0^T \lvert \omega\big(t,\eta(t,x)\big)\rvert \, dt = \infty.
\end{equation}
Then either there is a sequence $t_n$ of times such that $\eta(t_n)$ is conjugate to $\eta(t_{n+1})$ for every $n$ (in other words, $\eta$ fails to be a locally minimizing geodesic on the interval $[t_n,t_{n+1}]$), or there is a basis $\{e_1,e_2,e_3\}$ such that $\omega_0(x)$ is parallel to $e_3$ and the components $\Lambda_{ij}$ of $\Lambda(t) = D\eta(t,x)^{\dagger}D\eta(t,x)$ satisfy
\begin{equation}\label{conjugatefailure}
\int_0^T \frac{\Lambda_{33}(t)}{\Lambda^{11}(t)+\Lambda^{22}(t)} \, dt < \infty \qquad \text{and}\qquad \lim_{t\to T} \frac{\int_0^t \Lambda^{11}(\tau)\,d\tau}{\int_0^t \Lambda^{22}(\tau)\,d\tau} = 0.
\end{equation}
\end{theorem}

Conjugate points imply that the curvature is approaching positive infinity
, while the alternative condition allows for negative curvature but implies that the stretching matrix $\Lambda$ must have its eigenvectors aligning in fixed directions rather than rapidly rotating. In the present context we can get more explicit information along our trajectories
due to the symmetry.

\noindent
\textbf{Proof of Theorem \ref{conjugateblowupno}.}
The cases where $r_0=0$ and $r_0=1$ are similar, so we will just deal with $r_0=1$ here.
If $b_0$ is odd, then we have by \eqref{Detaexplicit} that
$$ \Lambda(t,1,0) = \left(\begin{matrix} \alpha_r^2 & 0 & 0 \\
0 & 1 & tb_3 \\
0 & t\mathfrak{b}_3 & t^2\mathfrak{b}_3^2 + \gamma_z^2\end{matrix}\right),$$
with the initial vorticity $\omega_0(1,0) = \mathfrak{b}_3 e_r$, where
$\mathfrak{b}_3 = (b_0)_z(1,0)$. Plugging $v(t) = f(t) e_r + g(t) e_{\theta} + h(t) e_z$  into
the index form \eqref{localizedindexform}, we get
$$ I(v,v) = \int_{t_1}^{t_2} \alpha_r^2 \dot{f}^2 + \big( \dot{g} + t\mathfrak{b}_3 \dot{h} + \mathfrak{b}_3 h\big)^2 + \gamma_z^2 \dot{h}^2 \, dt,$$
which is always positive.
Hence if the Jacobi field is spatially supported in a sufficiently small neighborhood of the circle $r=1$ and $z=0$, its index form in the volumorphism group will also be positive, and thus we cannot have conjugate points arising from locally-supported Jacobi fields.

By formula \eqref{pressurelapcomps}, using the fact that $P_r(t,1,0)=b(t,1,0)^2=0$ by the boundary condition
and $a(t,1,0) = a_z(t,1,0) = b(t,1,0)=0$ by our oddness assumptions and no-flow at the boundary, we get
\begin{equation}\label{oddpressurelaplacian}
\Delta P(t,1,0) = P_{rr}(t,1,0) + P_{zz}(t,1,0) = -2a_r(t,1,0)^2 = -2\,\frac{\alpha_{tr}(t,1,0)^2}{\alpha_r(t,1,0)^2}.
\end{equation}
Since the vorticity at $r=1$ and $z=0$ is $\omega(t,1,0) = -b_z(t,1,0) e_r = -\mathfrak{b}_3 \alpha_r(t,1,0) e_r$, our
assumption on the localized Beale-Kato-Majda criterion becomes $\int_0^T \alpha_r(t,1,0) \, dt = \infty$. Thus by
equation \eqref{oddpressurelaplacian} we have
$$ \int_0^T \Delta P(t,1,0) \, dt = -2 \lVert \dot{\xi}(t) \rVert^2_{L^2[0,T]} \le -C \sup_{0\le t\le T}\lvert \xi(t)\rvert^2,$$
where $\xi(t) = \ln{\alpha_r(t,1,0)}$, with the last inequality following from the Poincar\'e inequality on $[0,T]$. Hence
we must have $\int_0^T \Delta P(t,1,0)\,dt = -\infty$ or else $\xi(t)$ would be bounded, which is impossible since
$\int_0^T e^{\xi(t)}\,dt = \infty$ by assumption.

To show that the curvature also approaches negative infinity, it is enough to show that we can find a divergence-free
axisymmetric velocity field $V$ such that $K(U,V) \le C \Laplacian P(t,1,0)$ for some $C$ independent of time.
To do this we note that $K(U,V) \le C\int_M \nabla^2P(V,V) \, d\mu/\lVert V\rVert^2_{L^2(M)}$ by formula \eqref{gausscodazzi},
for some $C$ which depends on the $L^2$ norm of $U$ (which is constant in time). Thus it is enough to make $\int_M \nabla^2 P(V,V) \, d\mu$
close to $\Laplacian P(1,0)$. Let $\zeta$ be a smooth function on $\mathbb{R}$ supported in $(-1,1)$ which is odd, and
for a small $\varepsilon>0$, set $\psi(r,z) = \zeta(\frac{1-r}{\varepsilon}) \zeta(\frac{z}{\varepsilon})$. Then
$\psi$ is odd in $z$ and vanishes when $r=1$, so that the vector field
$V = -\frac{\psi_z}{r} e_r + \frac{\psi_r}{r} e_z$ is axisymmetric, divergence-free, with vertical component odd in $z$ and
tangent to the boundary $r=1$. It is easy to check that
$$ \int_M \nabla^2P(V,V) \, d\mu = \Big(\int_0^1 \zeta'(u)^2 \, du\Big) \Big(\int_0^1 \zeta(v)^2\,dv\Big)\Big(P_{rr}(1,0) + P_{zz}(1,0)\Big) + O(\varepsilon),$$
which is what we needed to show. Hence formula \eqref{infimumlaplacian} follows.\hfill$\square$

Roughly speaking, if the swirl component $b$ is even, then localized blowup may be associated to extreme
positive sectional curvature on the diffeomorphism group,
while if $b$ is odd, then it must be associated with negative curvature. Note that the condition
$\mathfrak{b}_1(\mathfrak{b}_1+\mathfrak{b}_2)>0$ is the same condition $b(r)\omega(r)>0$ at $r=1$ as the condition in \cite{washabaugh} for
an axisymmetric fluid flow $u = b(r) e_{\theta}$ to have positive sectional curvature in all directions,
so that it is natural that this condition yields infinitely many conjugate points. Curiously this
condition is also the \emph{opposite} of the condition $\mathfrak{b}_1(\mathfrak{b}_1+\mathfrak{b}_2)<0$ in
Theorem \ref{blowupiseasy} which ensures blowup.

\section{Outlook}
\label{outlook}

We have seen that there are several local criteria for the pressure Hessian that would ensure a ``local blowup,'' which is
to say a blowup of the vorticity along a single particle trajectory. Further we have analyzed the connection between this
possible blowup and the index form along the geodesic in the volumorphism group, inspired by the fact that the index form
involves precisely the quantities of vorticity and stretching which appear in the Beale-Kato-Majda criterion for blowup.
The reason the index form along the actual geodesic in the volumorphism group can be approximated by a local index form along
a single particle trajectory, and thus related to the supremum of vorticity, is precisely the lack of Fredholmness.

As such we would like to analyze other PDEs which happen to represent geodesic equations on infinite-dimensional groups,
which share the properties that their Riemannian exponential maps are smooth but not Fredholm. Such PDEs would be expected
to be good models of the phenomena we observe here for the axisymmetric 3D Euler equation. To date there are only two other
PDEs with these properties. In one dimension we have the Wunsch equation~\cite{wunsch}\cite{EKW}\cite{bauerkolevpreston}, for which all solutions
blow up in finite time~\cite{prestonwashabaugh} due to what appears to be positive sectional curvature everywhere. In
two dimensions the surface quasi-geostrophic equation (SQG) has a smooth exponential map which is not Fredholm, a recent result
of Washabaugh~\cite{washabaughSQG}. No global existence or blowup result is known for this equation, and even the analogue of
Theorem \ref{conjugatepoints} is not yet known, though it seems likely to be true. In fact all three equations are likely to have
very similar geometric structure.

\makeatletter \renewcommand{\@biblabel}[1]{\hfill#1.}\makeatother

\end{document}